\documentclass[11pt]{article}
\usepackage{amsmath, amscd, amssymb, latexsym, epsfig, color, amsthm,tikz, url}
\usepackage[all]{xy}
\usepackage{verbatim}
\usetikzlibrary{arrows,shapes,positioning,calc,decorations.pathreplacing}
\numberwithin{equation}{section}

\newtheorem{theorem}{Theorem}[section]
\newtheorem{proposition}[theorem]{Proposition}

\newtheorem{corollary}[theorem]{Corollary}
\newtheorem{conjecture}[theorem]{Conjecture}
\newtheorem{lemma}[theorem]{Lemma}

\newtheorem{problem}[theorem]{Problem}

\theoremstyle{definition}
\newtheorem{definition}[theorem]{Definition}
\newtheorem{example}[theorem]{Example}

\newtheorem{remark}[theorem]{Remark}

\DeclareMathOperator\lk{\mathrm{lk}}

\DeclareMathOperator\st{\mathrm{st}}

\newcommand{\R}{{\mathbb R}}

\newcommand{\kk}{{\mathbb K}}
\newcommand{\Bier}{\operatorname{Bier}}

\title{Sharp Bounds on the Independence Number of Simplicial Spheres}
\author{
	Jes\'us A. De Loera\thanks{Department of Mathematics, University of
	California, Davis, One Shields Avenue, Davis, CA 95616-8633, USA. Email:
	\texttt{deloera@math.ucdavis.edu}.}
	\qquad
	Ethan X. Fang\thanks{Department of Biostatistics \& Bioinformatics, Duke
	University, Durham, NC 27710, USA. Email: \texttt{ethan.fang@duke.edu}.}
	\qquad
	Junwei Lu\thanks{Department of Biostatistics, Harvard T.H. Chan School of
	Public Health, Boston, MA 02115, USA. Email: \texttt{junweilu@hsph.harvard.edu}.}
	\\[1ex]
	Tingzhou Wei\thanks{Department of Statistical Science, Duke University,
	Durham, NC 27708-0251, USA. Email: \texttt{tingzhou.wei@duke.edu}.}
	\qquad
	Hailun Zheng\thanks{Department of Mathematics, University of Hawai`i at
	M\={a}noa, 2565 McCarthy Mall, Honolulu, HI 96822, USA. Email:
	\texttt{hailunz@hawaii.edu}. }
}
\date{}
\begin{document}
	\maketitle
	\begin{abstract}
		We study the maximum size of an independent set in the graph of a simplicial sphere. Let $\beta(d,n)$ denote this maximum over all simplicial $(d-1)$-spheres on $n$ vertices, and let $\alpha(d,n)$ denote the maximum restricted to flag $(d-1)$-spheres. For every fixed $d\geq4$, we prove $\beta(d,n)=n-\Theta(n^{1/\lfloor d/2\rfloor})$. For flag spheres, we show $\alpha(d,n)\geq n-4\sqrt n+O(1)$ for all $d\geq4$ and determine the correct asymptotic order $\alpha(d,n)=n-\Theta(\sqrt n)$ for dimensions $d=4,5$. We also investigate the independence sets of Bier spheres and show that, in contrast to our other results, for this very large family of spheres, the independence number cannot be larger than $\left\lfloor\frac{n}{2}\right\rfloor.$
	\end{abstract}
\section{Introduction}
The independence number of a graph, the size of the largest set of pairwise nonadjacent vertices, is one of the most fundamental parameters studied in combinatorics and graph theory. The independence 
number of a graph ($\alpha(G)$ denotes the independence number of a graph $G$) is very important in graph theory as it is directly related to coloring, packing, covering, and clique problems. More generally, if $\Delta$ is a simplicial complex, we define its independence number of $\Delta$, $\alpha(\Delta)$, as the independence number of its $1$-skeleton $\Delta^{(1)}$. This paper focuses on the independence number of simplicial spheres.

Our research is justified for its impact in geometric, topological, and algebraic combinatorics.  First, it has interesting connections to Stanley-Reisner theory and complexes arising from graph theory \cite{Stanley1996,Kozlov2008,Jonsson2008}. Indeed, the \emph{independent complex} of any graph or complex $\Delta$  is the simplicial complex whose simplices are the independent sets of $\Delta$ and it has been studied in geometric-topological combinatorics by many authors \cite{EhrenborgHetyei2006,Kozlov2008}. Moreover, for a subset of vertices $W\subseteq V(\Delta)$, 
\[
    \Delta[W]=\{F\in\Delta:F\subseteq W\}
\]
denote the \emph{induced subcomplex} on $W$.
A set $W\subseteq V(\Delta)$ is independent in $\Delta^{(1)}$ if and only if
$\Delta[W]$ is a zero-dimensional complex. Now, suppose that $W$ is independent set of vertices and $|W|=r$.  Then
\[
    \widetilde H_0(\Delta[W];\kk)\cong \kk^{\,r-1}.
\]
Hochster's formula for the Stanley--Reisner ring $\kk[\Delta]$
\cite{Hochster1977,Stanley1996} states that
\[
 \beta_{p,q}\bigl(\kk[\Delta]\bigr)
 =\sum_{\substack{U\subseteq V(\Delta)\\ |U|=q}}
   \dim_{\kk}\widetilde H_{q-p-1}\bigl(\Delta[U];\kk\bigr).
\]
Taking $U=W$, $q=r$, and $p=r-1$ gives
\[
    \beta_{r-1,r}\bigl(\kk[\Delta]\bigr)\ge r-1.
\]
More strongly, every $s$-element subset of $W$ is independent, and hence
\[
    \beta_{s-1,s}\bigl(\kk[\Delta]\bigr)
       \ge \binom{r}{s}(s-1),
       \qquad 2\le s\le r.
\]
Thus, a large independent set forces a long nonzero linear strand in the
minimal free resolution of $\kk[\Delta]$.  Graph sparsity therefore has a
direct algebraic signature.  This implication is one-way: nonzero Betti
numbers can also arise from disconnected induced subcomplexes that are not
discrete.

Second, as mentioned earlier, the vertices of a given color form an independent set. If $N=|V(\Delta)|$, every proper coloring partitions the vertex set into independent color classes. Hence
\[
    \chi\bigl(\Delta^{(1)}\bigr)
    \ge \left\lceil\frac{N}{\alpha(\Delta)}\right\rceil.
\]
A pure $(d-1)$-dimensional simplicial complex is \emph{balanced} when its graph is properly $d$-colorable, balanced complexes have been extensively studied (see \cite{KleeNovik2016} and references therein). Balancedness therefore requires
\[
    \alpha(\Delta)\ge \left\lceil\frac{N}{d}\right\rceil.
\]
The independence number consequently detects obstructions to balancedness
and measures how far a complex may be from admitting a dimension-minimal
vertex coloring.  This numerical condition is necessary, but it is not by
itself sufficient for balancedness.

Finally, the independence number is particularly important to investigate for structured classes such as triangulated spheres and polytopal complexes, i.e., $\Delta=\partial P$ is the boundary complex of a simplicial polytope or sphere
$P$.  Under duality/polarity, a vertex $v$ of $P$ corresponds to a facet $F_v$ of the dual simple polytope $P^*$.  Two vertices of $P$ are adjacent precisely when the corresponding facets of $P^*$ meet. Therefore an independent set in
$\Delta^{(1)}$ corresponds to a collection of pairwise disjoint facets of
$P^*$.  In the polytopal setting, $\alpha(\Delta)$ is thus a natural
facet-packing invariant. 

In the setting of geometric-algebraic combinatorics, Chudnovsky and Nevo \cite{CNevo} proposed studying the independence number of simplicial spheres.
Recall that a simplicial complex is {\em flag} if every minimal nonface has cardinality two,
or equivalently, it is the clique complex of its graph \cite{Kozlov2008}.
Flag complexes arise naturally throughout combinatorics: the barycentric subdivision of any simplicial complex and the matching complex of any graph are classical examples. Geometrically, flag complexes arise naturally in Coxeter theory, cubical complexes, and criteria for nonpositive curvature \cite{KleeNovik2013}. Flag complexes are important because their entire higher-dimensional structure is determined by their 1-skeleton: whenever a collection of vertices is pairwise adjacent, it must form a face. Note also independence complexes of graphs are flag complexes, since $I(G)$ is the clique complex of the complement graph $\overline G$. Algebraically, a complex is flag precisely when its Stanley–Reisner ideal is generated by square-free quadratic monomials, connecting flagness with quadratic algebras and syzygies \cite{Stanley1996}.  In this paper we study the following two closely related problems.

\begin{problem}\label{prob:alphabeta}
	Determine how large an independent set a simplicial $(d-1)$-sphere on $n$ vertices can be. More specifically, determine
	\begin{align*}
		\alpha(d,n) &= \max\{\alpha(G) : |V(G)|=n,\ \text{the clique complex of $G$ is a flag
			$(d-1)$-sphere}\},\\
		\beta(d,n)  &= \max\{\alpha(G) : |V(G)|=n,\ G\text{ is the graph of a simplicial
			$(d-1)$-sphere}\}.
	\end{align*}
\end{problem}
In particular, $\alpha(d,n)\leq \beta(d,n)$.

Without the flag condition, large independent sets arise from stellar
subdivisions \cite{CNevo}: stellar subdividing a suitable collection of facets
of the boundary of a cyclic $d$-polytope yields
$\beta(d,n)\geq n-O(n^{1/\lfloor d/2\rfloor})$.
In this paper we prove the matching upper bound
for $\beta(d, n)$:

\begin{theorem}\label{thm:beta}
	For every fixed $d\geq 4$, as $n\to\infty$,
	\[
	\beta(d,n)=n-\Theta(n^{1/\lfloor d/2\rfloor}).
	\]
\end{theorem}

The flag condition was expected to impose a fundamentally stronger constraint on the independence number. Chudnovsky and Nevo proved that
	$\alpha(d,n)=\lfloor n/2\rfloor-(d-2)$ for dimensions $d=2,3$. They conjectured that this formula holds for every $d\geq 2$ \cite{CNevo}. Shah recently disproved their conjecture \cite{Shah}: for every $d\geq 4$, he constructed arbitrarily large flag $(d-1)$-spheres on $n$ vertices with independence number at least
	\[
	n-\frac{c n}{(\log n)^{\lfloor d/2\rfloor-1}},
	\]
	where $c=c(d)>0$. In each dimension, his construction starts with a
	neighborly cubical sphere of the corresponding dimension and replaces
	each cubical facet by a compatible flag triangulation after adding an
	interior vertex.

In this paper we further improve the lower bound and determine the correct asymptotic order of $\alpha(d,n)$ for $d=4$ and $d=5$. Our approach is the flag analog of the stellar subdivision construction above: instead of subdividing facets, we perform edge subdivisions (an operation that preserves flagness) on carefully chosen edges of a flag sphere. By a result of Zheng \cite{Z-flag} (see also \cite{Adam-Hladky} for sufficiently large odd-dimensional flag manifolds), the face numbers of flag $3$-spheres are simultaneously maximized by the join of two cycles of roughly equal length, suggesting that the natural flag substitute for the cyclic polytope is such a join. Our construction confirms this intuition.

\begin{theorem}\label{thm:lower}
	For all $d\geq4$,
	\[
	\alpha(d,n)\geq n-4\sqrt n+O(1).
	\]
Furthermore, for $d=4$ and $d=5$,
	\[
	\alpha(d,n)\leq n-\sqrt{12n}+O(1).
	\]
	In particular, $\alpha(4,n)=\alpha(5,n)=n-\Theta(\sqrt{n})$.
\end{theorem}

The proof of the upper bound on $\alpha(4, n)$ and $\alpha(5, n)$ relies on the Dehn--Sommerville relations, yielding a sharper estimate than the Kruskal--Katona bound used for general spheres. Theorem~\ref{thm:lower} shows that for these dimensions the flag condition does \emph{not} change the asymptotic order 
of the maximum independence number relative to general simplicial spheres.
For $d\geq 6$, a gap between the lower bound $n-O(n^{1/2})$ and the upper bound $n-\Omega(n^{1/\lfloor d/2\rfloor})$ remains open.

To conclude, one very important family of simplicial spheres is the \emph{Bier spheres}. These spheres arise as the deleted join of a simplicial complex $\Delta$ and its combinatorial Alexander dual $\Delta^\vee$; see \cite{BjornerPaffenholzSjostrandZiegler2005}. Since the input for this construction is an arbitrary simplicial complex $\Delta$, one can construct a sphere which contains $\Delta$ and shares with $\Delta$ many of its properties. For example, Bier spheres are known to contain copies of all the independence sets of $\Delta$, so in a way Bier spheres can serve to expose the \emph{independent complex} of any graph or complex $\Delta$. Another interesting property is that $\Bier(\Delta)$ is flag if and only if both $\Delta$ and $\Delta^\vee$ are flag \cite{HeudtlassKatthan2012}.
Bier spheres are also more abundant than simplicial polytopes (see \cite{BjornerPaffenholzSjostrandZiegler2005}). Here we prove

\begin{theorem}\label{thm:maximum-alpha}
Let $\Gamma=\Bier(\Delta)$ be a Bier sphere with $n\ge 3$  vertices.  Then
   $$ \alpha(\Gamma)\le \left\lfloor\frac{n}{2}\right\rfloor.$$
This bound is sharp for every $N\ge 3$.  Consequently, among Bier spheres on
$n$ vertices, the maximum independence set can have size $\left\lfloor\frac{n}{2}\right\rfloor.$
\end{theorem}

\noindent {\bf Structure of the paper:} Section~2 collects the necessary background on simplicial complexes and face numbers. In Section~3 we prove Theorem~\ref{thm:beta}, establishing sharp bounds on $\beta(d,n)$ for all $d\geq4$. Section~4 proves the lower bound on $\alpha(d,n)$, determines the exact independence number of the principal construction, and proves the upper bound for $d=4,5$. Section~5 discusses details on the independence sets and independence numbers of Bier spheres including a proof of Theorem \ref{thm:maximum-alpha}. Finally, Section~6 states the sharp-constant conjecture; explicit evidence is recorded in Appendix~\ref{app:computational}.

\section{Basic notation}
\subsection{Simplicial complexes and the face numbers}
A {\em simplicial complex} $\Delta$ with vertex set $V=V(\Delta)$ is a nonempty collection of subsets of $V$, called {\em faces}, that is closed under inclusion and contains all singletons: $\{v\}\in \Delta$ for all $v\in V$. A face $F\in \Delta$ has {\em dimension} $i$ if $|F| = i + 1$; in this case we say that $F$ is an $i$-face. The 0-faces, 1-faces, and inclusion-maximal faces are called {\em vertices}, {\em edges}, and {\em facets}, respectively. For brevity, we denote a vertex by $v$ and an edge by $uv$ rather than by $\{v\}$ and $\{u,v\}$. The {\em dimension} of $\Delta$ is $\max\{\dim F: F\in \Delta\}$.

If $F$ is a face of $\Delta$, then the {\em star} of $F$ and the {\em link} of $F$ are given by
	\[
	\st(F, \Delta)=\{G: F\cup G\in \Delta\},\quad
	\lk(F, \Delta)=\{G\in \st(F,\Delta): G\cap F=\emptyset\}.
	\]
If $W$ is a subset of $V(\Delta)$, then we let $\Delta\backslash W=\{G\in \Delta: G\cap W=\emptyset\}$. If $\Delta_1$ and $\Delta_2$ are two simplicial complexes on disjoint vertex sets, then the {\em join} of $\Delta_1$ and $\Delta_2$ is
\[
\Delta_1*\Delta_2=\{F\cup G: F\in \Delta_1, G\in \Delta_2\}.
\]

A simplicial complex is called {\em flag} if it is the clique complex of its graph. A simplicial complex $\Delta$ is called a {\em simplicial $(d-1)$-sphere} if its geometric realization is homeomorphic to a $(d-1)$-dimensional sphere. A basic example of a flag $(d-1)$-sphere is the octahedral $(d-1)$-sphere, the join of $d$ copies of a 0-sphere. We denote it by $O_d$. If $\Delta$ is flag, then it satisfies the {\em link condition}: for every edge $uv\in \Delta$,
	$\lk(u,\Delta)\cap\lk(v,\Delta)=\lk(uv,\Delta)$.

Let $\Delta$ be a simplicial complex of dimension $d-1$. Denote by $f_i=f_i(\Delta)$ the number of $i$-faces of $\Delta$, with the convention $f_{-1}(\Delta)=1$. The {\em $h$-numbers} of $\Delta$ are defined as a linear transformation of the $f$-numbers:
\[
h_j=h_j(\Delta)=\sum_{i=0}^j (-1)^{j-i}
\binom{d-i}{d-j}f_{i-1}(\Delta), \quad 0\leq j\leq d.
\]
It is easy to check that all the $f$-numbers are nonnegative integer combinations of the $h$-numbers. Furthermore,  if $\Delta$ is a simplicial $(d-1)$-sphere, then the $h$-numbers of $\Delta$ satisfy the Dehn--Sommerville relations $h_i=h_{d-i}$. This motivates the definition of the $g$-numbers: 
if we let $g_0=1$, and $g_i=h_i-h_{i-1}$ for all $1\leq i\leq \lfloor d/2\rfloor$, then the $g$-numbers completely determine the $h$- and $f$-numbers. 

We will need the following lemma in later sections:
\begin{lemma}\label{lm:f-h}
		Let $\Delta$ be a simplicial $(d-1)$-sphere. Then
	\[
	f_{d-2}(\Delta)=\begin{cases}
			(2m+1)\sum_{j=0}^m h_j(\Delta)& d=2m+1\\
			2m\sum_{j=0}^{m-1} h_j(\Delta)+mh_m(\Delta) & d=2m 
		\end{cases}.
	\]
\end{lemma}
\begin{proof}
	Recall that  for a simplicial $(d-1)$-sphere $\Delta$, $f_{i-1}(\Delta) = \sum_{j=0}^{i} \binom{d-j}{i-j} h_j(\Delta)$ for $0\leq i \leq d$.
	In the case $i=d-1$, we obtain $f_{d-2}(\Delta) = \sum_{j=0}^{d} (d-j)\, h_j(\Delta)$. We now simplify using the Dehn--Sommerville relations
	$h_j(\Delta)=h_{d-j}(\Delta)$. When $d=2m+1$, 
	\begin{align*}
	f_{d-2}(\Delta)&= \sum_{j=0}^{m} (2m+1-j)h_j(\Delta)+ \sum_{j=m+1}^{2m+1} (2m+1-j)h_j(\Delta)\\
	&=\sum_{j=0}^{m} (2m+1-j)h_j(\Delta)+\sum_{k=0}^{m} kh_k(\Delta)=(2m+1)\sum_{j=0}^m h_j(\Delta).
	\end{align*}
	When $d=2m$, 
	\begin{align*}
		f_{d-2}(\Delta)&= \sum_{j=0}^{m-1} (2m-j)h_j + mh_m
		+ \sum_{j=m+1}^{2m} (2m-j)h_j\\
		&=\sum_{j=0}^{m-1} (2m-j)\,h_j + m\,h_m + \sum_{k=0}^{m-1} k\,h_k
		= 2m\sum_{j=0}^{m-1} h_j(\Delta) + m\,h_m(\Delta).
	\end{align*}
\end{proof}

Throughout the paper, asymptotic notation refers to $n\to\infty$ with
$d$ fixed; implicit constants may depend on $d$.
The following exact bound on $\alpha(3,n)$ was proved in \cite{CNevo}.
It is attained by the suspension of an $(n-2)$-cycle.
\begin{theorem}
	For $n\geq 6$, $\alpha(3, n)=\lfloor n/2\rfloor -1$. 
\end{theorem}

\subsection{The upper and lower bounds on the face numbers}
To estimate the size of an independent set in a simplicial sphere, we use several standard bounds on face numbers. Recall that the cyclic $d$-polytope $C(d,n)$ is the convex hull of $n$ distinct points on the moment curve
$\{(t,t^2,\dots,t^d)\in\R^d:t\in\R\}$. Its boundary complex $\partial C(d,n)$ is $\lfloor d/2\rfloor$-neighborly: every subset of at most $\lfloor d/2\rfloor$ vertices is a face; see \cite{Ziegler}. The following lemma combines the Upper Bound Theorem \cite{Stanley75} with the nonnegativity of the $g$-vector for simplicial spheres \cite{Adiprasito-g-conjecture,KaruXiao}.
\begin{lemma}\label{lm:face-bounds}
	Let $d\geq 3$. Let $\Delta$ be a simplicial $(d-1)$-sphere with $n$ vertices. Then
	\begin{enumerate}
		\item $f_i(\Delta)\leq f_i(\partial C(d, n))$ for all $1\leq i\leq d-1$.
		\item $g_i(\Delta)\geq 0$ for $1\leq i\leq \lfloor d/2\rfloor$.
	\end{enumerate} 
\end{lemma}

For flag spheres, the octahedral sphere gives componentwise lower bounds on the face numbers \cite{Athanasiadis11}, but sharp upper bounds and stronger lower-bound conjectures remain open in most dimensions. Zheng proved the flag upper bound theorem in dimension 3 \cite{Z-flag}; Adamaszek and Hladk\'y obtained an upper bound theorem for odd-dimensional flag manifolds with sufficiently many vertices \cite{Adam-Hladky}. The Davis--Okun theorem \cite{DavisOkun} implies the required lower bound on the number of edges of a flag 3-sphere; the corresponding bound for flag 4-spheres follows by applying it to vertex links and using the Dehn--Sommerville relations.
\begin{lemma}
	Let $\Delta$ be a flag $(d-1)$-sphere. 
	\begin{enumerate}
		\item $f_i(\Delta)\geq f_i(O_{d})$ for all $0\leq i \leq d-1$.
		\item For $d=4$ and $5$, $f_1(\Delta)-(2d-3)f_0(\Delta)+2d(d-2)\geq 0$.
		\item For $d=4$, $f_1(\Delta)\leq \lfloor f_0(\Delta)^2/4\rfloor +f_0(\Delta)$. Equality holds when $\Delta$ is the join of two cycles whose lengths are as equal as possible.
	\end{enumerate}
\end{lemma}

\section{The bounds on $\beta(d, n)$}
In this section, we give a sharp bound on $\beta(d, n)$ for all $d\geq 4$. Our strategy is to estimate $f_{d-2}(\Delta\setminus I)$ in terms of $f_0(\Delta\setminus I)$, thereby deriving an upper bound on $I$. 

We begin with the lower bound on $\beta(d,n)$. The following construction is based on an example in \cite{CNevo}.
\begin{definition}
	Let $\Sigma(d,m)$ be the simplicial $(d-1)$-sphere obtained from
	$\partial C(d,m)$ by stellar subdividing all of its facets. For each
	facet $F$ of $\partial C(d,m)$, denote the new vertex by $x_F$.
\end{definition}
\begin{lemma}
	For every fixed $d\geq4$,
	\[
	\beta(d,n)\geq n-O(n^{1/\lfloor d/2\rfloor}).
	\]
\end{lemma}
\begin{proof}
	Consider $\Sigma(d,m)$. The set
	$I=\{x_F:F\text{ is a facet of }\partial C(d,m)\}$ is independent.
	By the face-number formula for cyclic polytopes,
	\[
	|I|=f_{d-1}(\partial C(d,m))
	=\Theta(m^{\lfloor d/2\rfloor}).
	\]
	On the other hand,
	$f_0(\Sigma(d,m))=m+|I|$. Writing
	$n=f_0(\Sigma(d,m))$, we obtain
	\[
	\beta(d,n)\geq |I|
	=n-O(n^{1/\lfloor d/2\rfloor}),
	\]
	as desired.
\end{proof}
\begin{lemma}\label{thm:beta-upper}
	For every $d\geq4$, there is a constant $c_d>0$ such that
	\[
	\beta(d,n)\leq n-c_dn^{1/\lfloor d/2\rfloor}
	\]
	for all sufficiently large $n$.
\end{lemma}
\begin{proof}
	Let $\Delta$ be a simplicial $(d-1)$-sphere with $n$ vertices, and let
	$I$ be a maximum independent set. By Lemma~\ref{lm:f-h},
	\begin{equation}\label{eq:f-h}
	f_{d-2}(\Delta)=
	\begin{cases}
		(2m+1)\displaystyle\sum_{j=0}^m h_j(\Delta),&d=2m+1,\\[1ex]
		2m\displaystyle\sum_{j=0}^{m-1}h_j(\Delta)+mh_m(\Delta),&d=2m.
	\end{cases}
	\end{equation}
	Using
	$h_j(\Delta)=\sum_{i=0}^j(-1)^{j-i}
	\binom{d-i}{j-i}f_{i-1}(\Delta)$, we obtain
	\begin{equation}\label{eq: f_{d-2}}
	f_{d-2}(\Delta)=
	\begin{cases}
		(2m+1)\displaystyle\sum_{i=0}^m(-1)^{m-i}
		\binom{2m-i}{m-i}f_{i-1}(\Delta),&d=2m+1,\\[2ex]
		\displaystyle\sum_{i=1}^m(-1)^{m-i}i
		\binom{2m-i-1}{m-i}f_{i-1}(\Delta),&d=2m.
	\end{cases}
	\end{equation}
	Each $(d-2)$-face of $\Delta\setminus I$ belongs to two facets of
	$\Delta=(\Delta\setminus I)\cup\bigcup_{v\in I}\st(v,\Delta)$. Since
	each facet of $\Delta\setminus I$ contains $d$ such faces, it follows
	that
	\[
	2f_{d-2}(\Delta\setminus I)
	=df_{d-1}(\Delta\setminus I)
	+\sum_{v\in I}f_{d-2}(\lk(v,\Delta)).
	\]

	First assume that $d=2m$. Using
	$f_i(\Delta)=f_i(\Delta\setminus I)
	+\sum_{v\in I}f_{i-1}(\lk(v,\Delta))$ and
	$f_{d-1}(\Delta\setminus I)\geq0$, we obtain
	$$
	f_{d-2}(\Delta\setminus I)
	=f_{d-2}(\Delta)
	-\sum_{v\in I}f_{d-3}(\lk(v,\Delta))\geq\sum_{v\in I}\frac12f_{d-2}(\lk(v,\Delta)),$$
    By (\ref{eq: f_{d-2}}), this is equivalent to $$ \sum_{i=1}^m(-1)^{m-i}i
	\binom{2m-i-1}{m-i}f_{i-1}(\Delta)\geq\sum_{v\in I}
	\left(\frac12f_{d-2}+f_{d-3}\right)(\lk(v,\Delta)).
	$$
    Now apply $f_i(\Delta)=f_i(\Delta\backslash I)+\sum_{v\in I} f_{i-1}(\lk(v, \Delta))$ again:
	\begin{align}
	\sum_{i=1}^m(-1)^{m-i}i
	\binom{2m-i-1}{m-i}f_{i-1}(\Delta\setminus I)
	&\geq\sum_{v\in I}\left(
	\begin{gathered}
	\frac12f_{d-2}+f_{d-3}\\[-0.3ex]
	-\displaystyle\sum_{i=1}^m(-1)^{m-i}i
	\binom{2m-i-1}{m-i}f_{i-2}
	\end{gathered}
	\right)(\lk(v,\Delta))\notag\\
	&\geq\sum_{v\in I}\left(
	\begin{gathered}
	\frac12f_{d-2}+(m-1)g_{m-1}\\[-0.3ex]
	{}+\cdots+g_1
	\end{gathered}
	\right)(\lk(v,\Delta))\notag\\
	&\geq\frac d2|I|.\label{eq:d-even}
	\end{align}
	Here we used \eqref{eq:f-h} and Lemma~\ref{lm:face-bounds}: every
	vertex link is a simplicial $(d-2)$-sphere, has at least $d$ facets,
	and has nonnegative $g$-numbers.

	The case $d=2m+1$ is similar:
	\begin{align}
	(2m+1)\!\sum_{i=0}^m(-1)^{m-i}
	\binom{2m-i}{m-i}\!f_{i-1}(\Delta\!\setminus\! I)
	&\!\geq\!\sum_{v\in I}\!\left(
	\begin{gathered}
	\frac12f_{d-2}+f_{d-3}\\[-0.3ex]
	{}-(2m+1)\displaystyle\sum_{i=0}^m(-1)^{m-i}
	\binom{2m-i}{m-i}f_{i-2}
	\end{gathered}
	\right)(\lk(v,\Delta))\notag\\
	&\geq\sum_{v\in I}\left(
	\begin{gathered}
	\frac12f_{d-2}+mg_m+(m-1)g_{m-1}\\[-0.3ex]
	{}+\cdots+g_1
	\end{gathered}
	\right)(\lk(v,\Delta))\notag\\
	&\geq\frac d2|I|.\label{eq:d-odd}
	\end{align}

	In both cases, suppose to the contrary that
	$f_0(\Delta\setminus I)=o(n^{1/m})$, where
	$m=\lfloor d/2\rfloor$. Then
	$f_1(\Delta\setminus I)=o(n^{2/m})$. By the Kruskal--Katona theorem,
	\[
	f_{j-1}(\Delta\setminus I)=o(n^{j/m})
	\qquad(2\leq j\leq m).
	\]
	Thus the left-hand side of \eqref{eq:d-even} or \eqref{eq:d-odd} is
	$o(n)$, while the right-hand side is at least
	$\frac d2|I|=\Theta(n)$, a contradiction. Hence
	$f_0(\Delta\setminus I)\geq c_dn^{1/m}$ for some $c_d>0$, which is
	equivalent to the desired upper bound.
\end{proof}

The preceding lemmas immediately imply the main result of this section.
\begin{theorem}
	For every fixed $d\geq4$, as $n\to\infty$,
	\[
	\beta(d,n)=n-\Theta(n^{1/\lfloor d/2\rfloor}).
	\]
\end{theorem}
\begin{corollary}
	For every fixed $d\geq4$,
	\[
	\alpha(d,n)\leq
	n-\Theta(n^{1/\lfloor d/2\rfloor}).
	\]
\end{corollary}

\section{The bounds on $\alpha(d,n)$}
In this section, we study maximum independent sets in flag spheres. We first
construct flag spheres with large independent sets and then prove upper bounds
for $d=4,5$.

\subsection{The lower bound}

Let $\Sigma$ be a simplicial sphere and let $e$ be an edge. The {\em stellar
subdivision} of $\Sigma$ at $e$ replaces
$e*\lk(e,\Sigma)$ by $x_e*\partial e*\lk(e,\Sigma)$, where $x_e$ is a new
vertex. Stellar subdivision at an edge preserves flagness; see, for example,
\cite{LutzNevo}. Motivated by the flag upper bound theorem in dimension 3, we
start with a join of two cycles and subdivide a collection of edges between
them.

\begin{definition}
	Let $s\geq2d$. For $d=4$, let $\Delta_{4,s}$ be the join of a
	$\lfloor s/2\rfloor$-cycle and a $\lceil s/2\rceil$-cycle. For $d>4$,
	define
	\[
	\Delta_{d,s}
	=\Delta_{4,s-2(d-4)}*O_{d-4}.
	\]
	The two cycles in $\Delta_{d,s}$ have lengths
	\[
	a=\lfloor s/2\rfloor-(d-4),
	\qquad
	b=\lceil s/2\rceil-(d-4).
	\]
	Both lengths are at least four.
	Choose maximum independent sets $I_1$ and $I_2$ in these cycles and set
	\[
	E_{d,s}=\{xy:x\in I_1,\ y\in I_2\}.
	\]
	Let $\Gamma_{d,s}$ be the complex obtained from $\Delta_{d,s}$ by
	stellar subdividing all edges in $E_{d,s}$.
\end{definition}

\begin{lemma}
	The complex $\Gamma_{d,s}$ is a flag $(d-1)$-sphere with
	\[
	f_0(\Gamma_{d,s})=s+q_d(s),
	\qquad
	\alpha(\Gamma_{d,s})\geq q_d(s),
	\]
	where
	\[
	q_d(s)=\left\lfloor\frac{s-2d+8}{4}\right\rfloor
	 \left\lfloor\frac{s-2d+9}{4}\right\rfloor .
	\]
\end{lemma}
\begin{proof}
	The complex $\Delta_{d,s}$ is a flag $(d-1)$-sphere. No two distinct
	edges in $E_{d,s}$ lie in a common face: their endpoints in each cycle
	belong to an independent set. Consequently, the interiors of their
	stars are pairwise disjoint, and the new vertices
	$\{x_e:e\in E_{d,s}\}$ are pairwise nonadjacent. Since edge subdivision
	preserves both the sphere property and flagness, this proves the first
	two assertions.

	A cycle of length $\ell$ has independence number $\lfloor\ell/2\rfloor$.
	Therefore
	\[
	|I_1|=\left\lfloor\frac{s-2d+8}{4}\right\rfloor,
	\qquad
	|I_2|=\left\lfloor\frac{s-2d+9}{4}\right\rfloor,
	\]
	and the formula for $q_d(s)=|I_1||I_2|$ follows.
\end{proof}

\begin{theorem}
	For $d\geq4$,
	\[
	\alpha(d,n)\geq n-4\sqrt n-2d.
	\]
\end{theorem}
\begin{proof}
	Fix $d\geq4$ and consider the family of flag $(d-1)$-spheres
	$\Gamma_{d,s}$. Let $N=f_0(\Gamma_{d,s})$. Since
	$\alpha(\Gamma_{d,s})\geq N-s$, it suffices to show that
	$s\leq4\sqrt N+2d$. This follows immediately from the preceding formula
	for $q_d(s)$.
\end{proof}

\begin{proposition}\label{prop:exact-construction}
	For every $d\geq4$ and $t\geq2$,
	\[
	\alpha\!\left(\Gamma_{d,\,4t+2(d-4)}\right)=t^2.
	\]
\end{proposition}
\begin{proof}
	First consider $d=4$. Write the two $2t$-cycles in alternating order as
	\[
	a_1,a'_1,a_2,a'_2,\dots,a_t,a'_t
	\quad\text{and}\quad
	b_1,b'_1,b_2,b'_2,\dots,b_t,b'_t,
	\]
	and subdivide every edge $a_ib_j$. Denote the new vertex by $x_{ij}$.
	The set $\{x_{ij}:1\leq i,j\leq t\}$ is independent, so the independence
	number is at least $t^2$.

	For the reverse inequality, indices are read modulo $t$. Cover the
	vertices by the following $t^2$ cliques:
	\[
	\{x_{ii},a_i,a'_i,b'_i\},\qquad
	\{x_{i+1,i},b_i\}\qquad(1\leq i\leq t),
	\]
	together with one singleton clique for every remaining $x_{ij}$.
	An independent set meets each clique in at most one vertex, so its size
	is at most $t^2$.

	For $d>4$, edge subdivision commutes with taking the join with
	$O_{d-4}$, and hence
	\[
	\Gamma_{d,\,4t+2(d-4)}=\Gamma_{4,4t}*O_{d-4}.
	\]
	The independence number of a join is the maximum of the independence
	numbers of its factors. Since $\alpha(O_{d-4})=2\leq t^2$, the result
	follows.
\end{proof}

\subsection{The case of $\alpha(4, n)$ and $\alpha(5,n)$}
We now turn to upper bounds on $\alpha(d,n)$. Throughout this subsection,
whenever $\Delta$ is a flag $(d-1)$-sphere and $I$ is an independent set, we
write $A=\Delta\setminus I$. We begin with an inequality valid for every flag
$(d-1)$-sphere.

\begin{lemma}\label{lm:general-inequality}
	Let $\Delta$ be a flag $(d-1)$-sphere and let $I$ be an independent set.
	Then
	\[
	\frac{4f_1(A)^2}{f_0(A)}
	\leq f_0(A)f_1(A)+3f_2(A).
	\]
\end{lemma}
\begin{proof}
	The induced subcomplex $A$ is flag and hence satisfies the link condition.
	For every edge $uv\in A$,
	\begin{equation}\label{eq:key}
		f_0(\lk(u,A))+f_0(\lk(v,A))
		=f_0(\lk(u,A)\cup\lk(v,A))+f_0(\lk(uv,A)).
	\end{equation}
	Summing \eqref{eq:key} over all edges of $A$ and using
	$f_0(\lk(u,A)\cup\lk(v,A))\leq f_0(A)$ gives
	\[
	\sum_{u\in V(A)}f_0(\lk(u,A))^2
	\leq f_0(A)f_1(A)+3f_2(A).
	\]
	Since $\sum_{u\in V(A)}f_0(\lk(u,A))=2f_1(A)$, the result follows
	from the Cauchy--Schwarz inequality.
\end{proof}

For $d=4,5$, the Dehn--Sommerville relations provide enough control over
$f_2(A)$ to improve on the Kruskal--Katona estimate used in
Lemma~\ref{thm:beta-upper}.

\begin{lemma}\label{lm:edge-bound}
	Let $d=4$ or $5$, let $\Delta$ be a flag $(d-1)$-sphere, and let $I$
	be an independent set. Then
	\[
	4f_1(A)\leq
	\begin{cases}
		(f_0(A)+6)f_0(A),&d=4,\\
		(f_0(A)+12)f_0(A),&d=5.
	\end{cases}
	\]
\end{lemma}
\begin{proof}
	First assume $d=4$. For $u\in V(A)$, the complex $\lk(u,A)$ is
	obtained from the flag 2-sphere $\lk(u,\Delta)$ by deleting the
	independent set $I\cap V(\lk(u,\Delta))$. Every edge link in a flag
	3-sphere is a cycle of length at least four. Therefore
	\begin{align*}
	f_1(\lk(u,A))
	&=3f_0(\lk(u,\Delta))-6
	-\sum_{\substack{v\in I\\uv\in\Delta}}f_0(\lk(uv,\Delta))\\
	&\leq3f_0(\lk(u,A))-6-|I\cap V(\lk(u,\Delta))|.
	\end{align*}
	Summing over $u\in V(A)$ gives
	\[
	3f_2(A)
	\leq6(f_1(A)-f_0(A))
	-\sum_{v\in I}f_0(\lk(v,\Delta))
	\leq6f_1(A).
	\]
	Lemma~\ref{lm:general-inequality} now yields
	$4f_1(A)\leq(f_0(A)+6)f_0(A)$.

	Now assume $d=5$ and set $n=f_0(\Delta)$. The Dehn--Sommerville
	relations give $f_2(\Delta)=4f_1(\Delta)-10f_0(\Delta)+20$. Hence
	\begin{align}
	f_2(A)
	&=f_2(\Delta)-\sum_{v\in I}f_1(\lk(v,\Delta))\notag\\
	&=4f_1(A)-10n+20
	-\sum_{v\in I}(f_1-4f_0)(\lk(v,\Delta))\notag\\
	&\leq4f_1(A)-10n+20
	-\sum_{v\in I}f_0(\lk(v,\Delta))+16|I|\notag\\
	&\leq4f_1(A)-10n+20+8|I|
	\leq4f_1(A).\label{eq:d5-f2}
	\end{align}
	The first inequality uses the Davis--Okun inequality
	$\gamma_2=f_1-5f_0+16\geq0$ for flag 3-spheres. The second uses
	$f_0(\lk(v,\Delta))\geq8$, and the last uses $|I|\leq n$ and
	$n\geq10$. Applying Lemma~\ref{lm:general-inequality} to
	\eqref{eq:d5-f2} gives
	$4f_1(A)\leq(f_0(A)+12)f_0(A)$.
\end{proof}

\begin{theorem}
	For $d=4,5$,
	\[
	\alpha(d,n)\leq n-\sqrt{12n}+O(1).
	\]
\end{theorem}
\begin{proof}
	Let $\Delta$ be an $n$-vertex flag $(d-1)$-sphere and let $I$ be a
	maximum independent set.

	First let $d=4$. Every flag 2-sphere has at least eight facets and
	$g_1\geq2$. Thus \eqref{eq:d-even} gives
	\[
	2f_1(A)-2f_0(A)
	\geq\sum_{v\in I}\left(\frac12f_2+g_1\right)(\lk(v,\Delta))
	\geq6|I|.
	\]
	Hence $f_1(A)\geq f_0(A)+3|I|=n+2|I|$. Combining this with
	Lemma~\ref{lm:edge-bound} and solving the resulting quadratic
	inequality gives
	\[
	|I|\leq n+7-\sqrt{12n+49}.
	\]

	Now let $d=5$. For a flag 3-sphere $L$, we have
	$f_3(L)\geq16$, $g_1(L)=f_0(L)-5\geq3$, and
	\[
	g_2(L)=\gamma_2(L)+f_0(L)-6\geq2
	\]
	by the Davis--Okun inequality. It follows from
	\eqref{eq:d-odd} that
	\begin{align*}
	5(f_1(A)-3f_0(A)+6)
	&\geq\sum_{v\in I}
	\left(\frac12f_3+2g_2+g_1\right)(\lk(v,\Delta))\\
	&\geq15|I|.
	\end{align*}
	Therefore $f_1(A)\geq3f_0(A)-6+3|I|=3n-6$.
	Lemma~\ref{lm:edge-bound} now gives
	\[
	4(3n-6)\leq f_0(A)(f_0(A)+12).
	\]
	Thus $f_0(A)\geq\sqrt{12n+12}-6$, or equivalently,
	$|I|\leq n+6-\sqrt{12n+12}$.
\end{proof}

\section{The case of Bier spheres}

Bier spheres are constructed as follows. Let
$E=[m]=\{1,\dots,m\}$ be a ground set, and let $K\subsetneq 2^E$ be a
simplicial complex with $E\notin K$. Recall that its \emph{Alexander dual} is
\[
    K^\vee
    =\bigl\{B\subseteq E:E\setminus B\notin K\bigr\}.
\]
Let $\overline E=\{\bar 1,\dots,\bar m\}$ be a disjoint barred copy of $E$.

\begin{definition}
The \emph{Bier complex} of $K$ is the deleted join
\[
    \Bier(K)=K*_{\Delta}K^\vee.
\]
More explicitly, its faces are
\[
    A\sqcup\overline B,
    \qquad
    A\in K,\quad B\in K^\vee,\quad A\cap B=\varnothing,
\]
where $\overline B=\{\bar i:i\in B\}$.
\end{definition}

Bier's theorem asserts that $\Bier(K)$ is a combinatorial sphere of dimension
$m-2$; see \cite{BjornerPaffenholzSjostrandZiegler2005}.  Some of the $2m$
formal labels may fail to be actual vertices:
\[
 \begin{aligned}
   i\in V\bigl(\Bier(K)\bigr)
       &\Longleftrightarrow \{i\}\in K,\\
   \bar i\in V\bigl(\Bier(K)\bigr)
       &\Longleftrightarrow \{i\}\in K^\vee
        \Longleftrightarrow E\setminus\{i\}\notin K.
 \end{aligned}
\]
We always count actual vertices rather than absent or ``ghost'' labels.

Note that when we construct Bier spheres from a complex $K$ one uses two disjoint copies of the ground set [m]:
\[
[m]^+=\{1,\ldots,m\},
\qquad
[m]^-=\{\bar1,\ldots,\bar m\}.
\]
We use the following terminology for one of the two sides of the deleted join:
the \emph{positive or unbarred shore} consists of the vertices $i^+$ coming from $K$.
The \emph{negative or barred shore} consists of the vertices $\bar i$ coming from $K^\vee$.

The graph of a Bier sphere has a particularly transparent form and we can calculate explicitly its independence numbers.

\begin{proposition}\label{prop:bier-edges}
Whenever the relevant endpoints are actual vertices of $\Bier(K)$, its edges
are described as follows:
\begin{enumerate}
    \item $i$ and $j$ are adjacent if and only if $\{i,j\}\in K$;
    \item $\bar i$ and $\bar j$ are adjacent if and only if
          $\{i,j\}\in K^\vee$;
    \item $i$ and $\bar j$ are adjacent whenever $i\neq j$;
    \item $i$ and $\bar i$ are never adjacent.
\end{enumerate}
\end{proposition}

\begin{proof}
The first two statements follow directly from the two factors in the deleted
join.  A mixed pair $\{i,\bar j\}$ is a face precisely when the two underlying
labels are disjoint, which is equivalent to $i\neq j$.
\end{proof}

Let $G_K$ and $G_{K^\vee}$ denote the graphs of the two factors, restricted to
their actual vertices.  Proposition~\ref{prop:bier-edges} immediately gives
the following formula.

\begin{proposition}\label{prop:alpha-formula}
Define
\[
 \varepsilon(K)=
 \begin{cases}
  2,&\text{if both $i$ and $\bar i$ occur for some $i\in E$},\\
  0,&\text{otherwise}.
 \end{cases}
\]
Then
\[
 \alpha\bigl(\Bier(K)\bigr)
 =\max\bigl\{\alpha(G_K),\alpha(G_{K^\vee}),\varepsilon(K)\bigr\}.
\]
In particular, if all $2m$ vertices occur, then
\[
 \alpha\bigl(\Bier(K)\bigr)
 =\max\bigl\{2,\alpha(G_K),\alpha(G_{K^\vee})\bigr\}.
\]
\end{proposition}

\begin{proof}
An independent set contained in one shore is exactly an independent set in
$G_K$ or in $G_{K^\vee}$.  If an independent set uses both shores, then every
unbarred label and every barred label in it must have the same underlying
index.  Such a mixed independent set therefore has at most two vertices and
exists with two vertices precisely when some pair $\{i,\bar i\}$ occurs.
\end{proof}

\begin{example}\label{ex:ten-vertices}
The number of vertices alone does not determine the independence number.
Consider complexes on $E=[5]$.
\begin{enumerate}
    \item If $K=\{A\subseteq[5]:|A|\le 2\}$,
    then $K=K^\vee$.  Both shore are complete graphs, and the Bier graph is
    $K_{10}$ with the five pairs $\{i,\bar i\}$ removed.  Hence
    $\alpha\bigl(\Bier(K)\bigr)=2.$

    \item If $K=\{A\subseteq[5]:|A|\le 1\}$,
    then the five unbarred vertices are mutually nonadjacent.  All ten
    vertices occur, and $\alpha\bigl(\Bier(K)\bigr)=5.$
\end{enumerate}
Thus two Bier spheres on ten vertices can have independence numbers $2$ and
$5$, respectively.
\end{example}

Independence complexes $I(K)$ are very important in topological combinatorics \cite{Bjorner1995TopologicalMethods,Kozlov2008} and have been intensely studied (see e.g., \cite{EhrenborgHetyei2006} and references therein).
The relationship between the independence complexes of $K$ and those of its Bier sphere can be summarized as follows: Independent sets of the $\Bier (K)$, using only positive vertices (positive shore) are precisely copies of the independent sets of $K$: $\{A^+:A\in\mathcal I(K)\}.$ 
Independent sets of $\Bier (K)$, using only negative vertices (barred shore) are precisely the complements of independent sets of $K^\vee$:
    $\{\overline A:A\in\mathcal I(K^\vee)\}.$
 An independent set using both positive and negative vertices can only be a matched pair  $\{i^+,\bar i\}.$ This last assertion follows because
$\{i^+,\bar j\}$ is an edge of the Bier sphere whenever $i\neq j$, whereas 
$\{i^+,\bar i\}$ is not an edge. Therefore,

\[
\mathcal I(B(K))
=
\underbrace{\mathcal I(K)^+}_{\text{positive vertices}}
\cup
\underbrace{\overline{\mathcal I(K^\vee)}}_{\text{negative vertices}}
\cup
\underbrace{\bigcup_i 2^{\{i^+,\bar i\}}}_{\text{matched independent pairs}},
\]

where the union over $i$ includes only indices for which both vertices $i^+$ and $\bar{i}$ actually occur.

The Bier construction converts low-dimensional nonfaces of $K$ into
high-dimensional faces of its Alexander dual.  In particular, if $A$ is an
independent set in the unbarred shore, then for every distinct $i,j\in A$,
\[
    \{i,j\}\notin K
    \quad\Longleftrightarrow\quad
    E\setminus\{i,j\}\in K^\vee.
\]
Because $i$ and $j$ are actual vertices of $K$, the missing edge
$\{i,j\}$ is a minimal nonface of $K$.  Therefore,
$E\setminus\{i,j\}$ is in fact a facet of $K^\vee$.  Thus, an independent set
of size $r$ on one shore produces $\binom{r}{2}$ explicitly structured
$(m-3)$-dimensional facets on the Alexander-dual shore.  The analogous
statement holds with $K$ and $K^\vee$ interchanged.

\subsection{The largest independence number of a Bier sphere, proof of Theorem \ref{thm:maximum-alpha}}

\begin{proof}[of Theorem \ref{thm:maximum-alpha}]
We use $N$ for the number of actual vertices of the Bier sphere, thereby
distinguishing it from the size $m$ of the ground set of the original input simplicial complex. 

Let $I$ be an independent set of cardinality $r$.  If $r\le 2$, the desired
inequality is immediate for $N\ge4$; the case $N=3$ is the triangular
$1$-sphere and has independence number one.

Suppose now that $r\ge3$.  Proposition~\ref{prop:bier-edges} implies that $I$
lies entirely on one shore.  Assume first that
\[
    I=\{i_1,\dots,i_r\}
\]
lies on the unbarred shore.  Then $\{i_a,i_b\}\notin K$ whenever $a\neq b$.
For every $j\in E$, choose two distinct elements $i_a,i_b\in I\setminus\{j\}$;
this is possible because $r\ge3$.  If $E\setminus\{j\}$ belonged to $K$, then
its subset $\{i_a,i_b\}$ would also belong to $K$, a contradiction.  Hence
\[
    E\setminus\{j\}\notin K
\]
for every $j\in E$.  Equivalently, all $m$ barred vertices occur.  Since the
$r$ vertices of $I$ also occur and $m\ge r$, we obtain
\[
    N\ge m+r\ge 2r.
\]
Thus $r\le\lfloor N/2\rfloor$.  The argument for an independent set on the
barred shore is identical after interchanging $K$ and $K^\vee$.

It remains to show sharpness.  We give explicit constructions.

If $N=2m$ with $m\ge3$, take $K$ to be the zero-dimensional complex on $[m]$:
\[
    K=\{\varnothing,\{1\},\dots,\{m\}\}.
\]
All $2m$ Bier vertices occur, and the $m$ unbarred vertices form an independent
set.  Therefore $\alpha(\Bier(K))=m=N/2$.

If $N=2m-1$ with $m\ge3$, work on the ground set $E=[m]$ and take
\[
    K=\{A\subseteq[m]:|A|\le m-2\}
      \cup\bigl\{[m-1]\bigr\}.
\]
Its Alexander dual is
\[
    K^\vee
      =\bigl\{\varnothing,\{1\},\dots,\{m-1\}\bigr\}.
\]
Thus all $m$ unbarred vertices occur, exactly $m-1$ barred vertices occur,
and those barred vertices form an independent set.  It follows that
\[
    \alpha\bigl(\Bier(K)\bigr)
      =m-1=\left\lfloor\frac{N}{2}\right\rfloor.
\]

Finally, the triangular Bier sphere gives equality for $N=3$.  For $N=4$,
take $K$ on $\{1,2,3\}$ with facets $\{1,2\}$ and $\{1,3\}$.  Its Bier graph
is a $4$-cycle and hence has independence number two.
\end{proof}

\begin{remark}
The zero-dimensional Bier sphere $S^0$ has two isolated vertices, so its
independence number is $2$.  It is the exception to the formula in
Theorem~\ref{thm:maximum-alpha} when $N=2$.
\end{remark}

\begin{corollary}
If the original simplicial complex $K$ has ground set of size $m$ and all $2m$ Bier vertices occur, then
\[
    \alpha\bigl(\Bier(K)\bigr)\le m,
\]
and this is best possible.
\end{corollary}

\section{Concluding remarks}

For $d=4,5$, set $\delta_d(n)=n-\alpha(d,n)$. Our results give
\[
\sqrt{12}\leq
\liminf_{n\to\infty}\frac{\delta_d(n)}{\sqrt n}
\leq
\limsup_{n\to\infty}\frac{\delta_d(n)}{\sqrt n}
\leq4.
\]
Proposition~\ref{prop:exact-construction} shows that the constructed family has
deficit asymptotic to four times the square root of its number of vertices. We
conjecture that this family is asymptotically extremal.

\begin{conjecture}\label{conj:sharp}
	For $d=4,5$,
	\[
	\lim_{n\to\infty}\frac{n-\alpha(d,n)}{\sqrt n}=4.
	\]
\end{conjecture}

Appendix~\ref{app:computational} gives a certified flag 3-sphere on 21 vertices
with independence number 9, exceeding the value $\lfloor n/2\rfloor-2$
predicted in \cite{CNevo}.

For $d\geq6$, even the order of the deficit is undetermined. The construction gives
$n-\alpha(d,n)\leq4\sqrt n+O(1)$, whereas the only upper bound on $\alpha(d,n)$
available to us is the one inherited from $\beta(d,n)$ in Section~3, namely
$n-\alpha(d,n)=\Omega\!\left(n^{1/\lfloor d/2\rfloor}\right)$. These two orders first
fail to meet at $d=6$, which we regard as the central open case.

\section{Acknowledgements}
\noindent\textbf{The role of AI in this proof.}
GPT-5.5 Pro, accessed through the ChatGPT web interface, was used as a sounding
board in the early stages of this project and suggested the initial idea behind
the results in Section~4. Building on this suggestion, the authors independently
developed the more general results in Section~3 and refined the argument into
the form presented here. The final proofs were written and verified by the authors. 

Jes\'us De Loera was partially supported by grants NSF DMS-2434665 and NSF DMS-2348578. The research of Junwei Lu was partially supported by NSF DMS-2434664. Hailun Zheng was partially supported by NSF grant DMS-2535689.
\appendix
\section{Computational evidence}\label{app:computational}

\noindent\emph{An explicit example.}\quad
There is a flag $3$-sphere $\Lambda$ on $21$ vertices with
$f(\Lambda)=(21,93,144,72)$ and $\alpha(\Lambda)=9$. As
$9>8=\lfloor(21-4)/2\rfloor$, it refutes the prediction
$\alpha(4,n)=\lfloor n/2\rfloor-2$ of \cite{CNevo}. The value
$\alpha(\Lambda)=9$ is certified by an independent set of size $9$ and a
partition of $V(\Lambda)$ into $9$ cliques. A machine-readable certificate
records $22$ edge subdivisions and $9$ admissible edge contractions connecting
$\Lambda$ to the boundary of the $4$-dimensional cross-polytope. These moves
preserve flag PL spheres \cite{LutzNevo}. The ancillary file
\texttt{verify\_n21.py} contains the certificate and a self-contained verifier.
It uses only the Python standard library and may be run with
\texttt{python3 verify\_n21.py}.

	{\small
	\bibliography{paper-refs}
	\bibliographystyle{plain}
}
\end{document}